\documentclass[a4paper,11pt]{amsart}
\usepackage[all,cmtip]{xy}
\usepackage{amsmath,amsthm,amssymb, amsfonts, amscd}
\usepackage{calrsfs}
\usepackage[colorlinks=true,linkcolor=blue,citecolor=blue]{hyperref}
\usepackage{color}
\usepackage{a4wide}

\newtheorem{theorem}{Theorem}[section]
\newtheorem{proposition}[theorem]{Proposition}
\newtheorem{corollary}[theorem]{Corollary}

\theoremstyle{definition}

\DeclareMathOperator{\id}{id}
\DeclareMathOperator{\spanned}{span}

\title[Bilinear forms on Banach sequence spaces]{Extendibility of bilinear forms on Banach sequence spaces}
\author{Daniel Carando}\thanks{The first author was partially supported by UBACyT W746 and CONICET  PIP 0624.}

\address{Departamento de Matem\'{a}tica - Pab I,
Facultad de Cs. Exactas y Naturales, Universidad de Buenos Aires,
(1428) Buenos Aires, Argentina and  IMAS - CONICET} \email{dcarando@dm.uba.ar}

\author{Pablo Sevilla-Peris}

\address{Instituto Universitario de Matem\'atica Pura y Aplicada and DMA, ETSIAMN, Universitat Polit\`ecnica de Val\`encia, Valencia, Spain} \email{psevilla@mat.upv.es}
\thanks{The second author was supported by MICINN Project MTM2011-22417 }

\begin{document}
\begin{abstract}
We study Hahn-Banach extensions of multilinear forms defined on Banach sequence spaces. We characterize $c_0$ in terms of extension of
bilinear forms, and describe the Banach sequence spaces in which every bilinear form admits extensions to any superspace.
\end{abstract}

\maketitle

\section{Introduction}

One of the fundamental results in Functional Analysis is the  Hahn-Banach theorem. It was proved independently by Hahn in 1927 \cite{Ha27}
and by Banach in 1929 \cite{Ba29} (see also \cite[Chapitre~IV, \S 2]{Ba32}). In one of its forms, it states that if $X$ is subspace of a normed space $Z$, then every continuous, linear functional
$f: X \to \mathbb{K}$ can be extended to $Z$ preserving the norm. It soon became clear that a multilinear version of this result was not possible in general, and this started the search of
situations on which such multilinear extension theorems are possible. A particular positive result was given by Arens in 1951, where he showed how to extend the product on a
Banach algebra to its bidual and, also, how to extend bilinear operators defined on a couple of Banach spaces to their corresponding biduals \cite{Ar51, Ar51b}.
This is one of the lines to find extension theorems: given a space, find a superspace to which every multilinear mapping can be extended.
Aron and Berner went further on this line and showed in 1978 that every holomorphic function on a Banach space can be extended to an open subset of the bidual \cite{ArBe78}.

Another line is to fix a Banach space $X$ and consider the problem of extending bilinear forms defined on subspaces  of $X$. Maurey's extension theorem \cite[Corollary 12.23]{DiJaTo95} is classical example of this natural point of view, in which relevant advances have been obtained in the last years \cite{CaGaDePGSu,FePr10}.

A third way to face the extension problem is to find the bilinear mappings (on a fixed Banach space) that can be extended to every superspace. This was the point of
view taken by Grothendieck: in 1956 he showed in his th\'eor\`eme
fondamental \cite[page 60]{Gr53} that these are precisely those bilinear mappings factoring through Hilbert spaces via 2-summing operators. We say  that a bilinear form
$T:X \times  Y \rightarrow \mathbb{K}$ is \textit{extendible}  (see e.g. \cite{Ca99,CaGaJa01,JaPaPGVi07,KiRy98}) if for
all Banach spaces $E \supset X,\,  F\supset Y$, there exists a bilinear form defined on $E \times  F$
that extends $T$. Our aim, which can be framed in this last approach,  is to describe those spaces which enjoy a bilinear (or multilinear) Hahn-Banach theorem, in the sense
that every bilinear form is extendible.
Examples of such spaces are $A(\mathbb D)$, $H^\infty(\mathbb D)$, $\mathcal L^\infty$-spaces and \emph{Pisier} spaces, but a complete characterization is still unknown.
In this line, our main result is the following theroem, which solves the problem among Banach spaces with unconditional basis.

\begin{theorem}\label{main1}
 The only Banach space with an unconditional basis on which every bilinear form is extendible is $c_0$.
\end{theorem}

 This theorem will follow as a consequence of Theorem~\ref{extens con base} below. We also characterize the Banach sequence spaces satisfying a bilinear Hahn-Banach theorem as those
``between'' $c_0$ and $\ell_\infty$ (see Corollary~\ref{ultimo}). As a byproduct, we obtain a partial answer to the following open problem:
if a sequence $X_n$ of $n$-dimensional Banach spaces is uniformly complemented in some $\mathcal L_\infty$, must these spaces be uniformly isomorphic to $\ell_\infty^n$? Corollary~\ref{equivalencias-extendibles}  gives a positive answer for sections of a Banach sequence space (see Proposition~\ref{open problem}).

\subsection{Preliminaries}

We briefly collect here some basic definitions that will be used throughout the paper. We will consider real or complex Banach spaces, that will be denoted
$X$, $Y, \ldots $. Unless otherwise stated they will be assumed to be infinite dimensional. The duals will be denoted by $X^{*}$, $Y^{*}$, $\ldots$.
Given two Banach spaces $X$ and $Y$, we write $X\approx Y$ if they are isomorphic and $X \stackrel{1}{\approx} Y$ if they are isometrically isomorphic.
We refer to \cite{AlKa06,Pi86} for basic concepts and notations on Banach spaces.

\smallskip
We denote $\mathcal{L}^{2} (X , Y)$ for the Banach space of all scalar valued, continuous, bilinear mappings (in short bilinear forms) on $X\times Y$. We
write $\mathcal{L}^{2} (X)$ whenever $Y = X$.

\smallskip
The space of extendible bilinear  forms is denoted by $\mathcal{E}^{2} (X, Y)$. The \textit{extendible  norm}
\begin{multline*}
\Vert  T \Vert _{\mathcal{E}} = \Vert  T \Vert _{\mathcal{E}^{2} (X,Y)} := \inf \{c>0:   \text{ for all } W\supseteq X, Z\supseteq Y \text{ there is an extension of }T\\
\text{ to }W\times Z    \text{ with norm}\leq c\}
\end{multline*}
makes $\mathcal{E}^{2} (X, Y)$ a Banach space.
Since every $\ell_\infty(I)$ space is injective (in fact, has the metric extension property), every bilinear form on such spaces is extendible and the extendible and
uniform norm coincide. Moreover, a
bilinear form $T$ on $X\times Y$ is extendible if and only if it extends to $\ell_\infty(I)\times \ell_\infty(J)$, for some $\ell_\infty(I) \supset X$ and $\ell_\infty(J)\supset Y$.
The supremum defining the extendible norm can be taken only over the extensions to  $\ell_\infty(I)\times \ell_\infty(J)$.

\smallskip
We write $\mathcal{L}(X;Y)$ for the space of all (continuous, linear)  operators $u:X \rightarrow Y$. We denote by $\Pi_{1} (X;Y)$  the space of absolutely summing
operators, $\Gamma_{\infty}(X;Y)$ for the $\infty$-factorable and $\Delta_{2}(X;Y)$ for the 2-dominated. Their corresponding norms are, respectively,
$\pi_1$, $\gamma_\infty$ and $\delta_2$ (see \cite{DeFl93,DiJaTo95} for definitions and basic properties).

\smallskip
We are going to use the theory tensor products and operator ideals as presented in \cite{DeFl93}. We recall some notation and definitions for completeness. The projective tensor norm $\pi$
 is defined, for a $z$ in the tensor product $ X\otimes Y$,  by
$$
\pi(z)=\inf\left\{\sum_{j=1}^r \| x_j\|\ \|y_j\|\right\},
$$
where the infimum is taken over
all the representations of $z$ of the form  $z=\sum_{j=1}^r  x_j\otimes y_j$.
The right-injective associate of $\pi$ is denoted by $\pi \backslash$. This tensor norm is  greatest right-injective tensor norm and  makes the following inclusion an isometry:
$$   X \otimes_{\pi \backslash} Y  \overset 1 \hookrightarrow  X \otimes_\pi \ell_{\infty}(B_{Y^*}),$$ where $B_{Y^*}$ is the  unit ball of $Y ^*$
(see \cite[Theorem 20.7.]{DeFl93} for details). Likewise, the injective associate $/ \pi \backslash$  is the largest
injective tensor norm and is induced by the isometric inclusion
$$   X \otimes_{/ \pi \backslash} Y  \overset 1 \hookrightarrow  \ell_{\infty}(B_{X^*})\otimes_\pi \ell_{\infty}(B_{Y^*}).$$
The metric extension property of $\ell_\infty(I)$ spaces implies that extendible bilinear forms are precisely the $/ \pi \backslash$-continuous ones:$$ \mathcal{E}^{2} (X, Y)\overset 1 =  \left(X \otimes_{/ \pi \backslash} Y\right)^* .$$
We refer to  \cite{DeFl93,DiJaTo95} for all the basic (and not so basic) facts and any undefined notation on tensor norms and operator ideals.

\smallskip
Given a family $\{X_{n}\}_{n}$ of Banach spaces where  $\dim X_{n} = n$, we say that $X_n$ are $K$-uniformly complemented in $X$ if for each $n$ we have
a mapping $i_n:X_n\to X$ and $q_n:X\to X_n$ such that $q_n\circ i_n$ is the identity on $X_n$ and $\|i_n\|\ \|q_n\|\le K$. In this case we also say that $X$ contains $X_n$
uniformly complemented. We note that if $X$ contains uniform copies of $\ell_\infty^n$ (i.e., $\mathbb K^n$ with the sup norm), then
the $\ell_\infty^n$ are uniformly complemented since they are injective spaces.

\subsection{Banach sequence spaces}

By a Banach  sequence space (also known as K\"othe sequence space) we will mean a Banach space $X \subseteq
\mathbb{K}^{\mathbb{N}}$ of sequences in $\mathbb{K}$ such that
$\ell_{1} \subseteq X \subseteq \ell_{\infty}$ with norm one inclusions satisfying that if
$x \in \mathbb{K}^{\mathbb{N}}$ and $y \in X$ are such that $\vert x_{n} \vert \leq \vert y_{n} \vert$ for all $n \in \mathbb{N}$, then
$x$ belongs to $X$ and $\Vert x \Vert \leq \Vert y \Vert$.\\
If $X$ is a Banach sequence space, we  denote by $\{e_{n}\}_{n}$  the sequence of canonical vectors, which is always a 1-unconditional basic sequence.
We define $X_{N} = \spanned \{e_{1} , \ldots , e_{N}\}$ and
$X_{0} = \overline{\spanned} \{e_{n}\}_{n}$. This last space is usually referred to as the minimal kernel of $X$. Given $x \in X$ we write $x^{N}=(x_{1} , \ldots , x_{N})$.
There are inclusions $i_{N}^{X} : X_{N} \hookrightarrow X$ and
projections $\pi_{N}^{X}: X \rightarrow X_{N}$ given by $i_{N}^{X} (x_{1}, \ldots , x_{N}) = (x_{1}, \ldots , x_{N}, 0, 0, \ldots )$ and
$\pi_{N}^{X}(x)=x^{N}$. The inclusions are isometric and the projections have norm $1$. For the case $X=\ell_p$ ($1\le p \le \infty$), we write $\ell_p^N$ for $X_N$.

\smallskip
\noindent Given a Banach sequence space $X$, its K\"othe dual is defined as
\[
 X^{\times} = \{ (z_{n})_{n} \in \mathbb{K}^{\mathbb{N}} \colon \textstyle\sum_{n} \vert z_{n} x_{n} \vert < \infty \text{ for all } x \in X\} \,.
\]
With the norm $\Vert z \Vert_{X^{\times}} = \sup_{\Vert x \Vert_{X} \leq 1} \sum_{n} \vert z_{n} x_{n} \vert$ it is again a Banach sequence space.

\smallskip
Following \cite[1.d]{LibroLiTz2}, a Banach sequence space $X$
is said to be $r$-convex (with $1 \leq r < \infty$) if there
exists a constant $\kappa > 0$ such that for any choice $x_{1}, \dots, x_{N} \in X$ we have
\[
\bigg\| \bigg( \Big( \sum_{j=1}^{N} | x_{j}(k)|^{r} \Big)^{1/r} \bigg)_{k=1}^{\infty} \bigg\|_{X}
\leq \kappa \ \bigg( \sum_{j=1}^{N} \| x_{j}\|_{X}^{r} \bigg)^{1/r}.
\]
On the other hand, $X$ is $s$-concave (with $1 \leq s < \infty$) if there is a constant $\kappa > 0$ such that
\[
\bigg( \sum_{j=1}^{N} \| x_{j}\|_{X}^{s} \bigg)^{1/s}
\leq \kappa \ \bigg\| \bigg( \Big( \sum_{j=1}^{N} | x_{j}(k)|^{s} \Big)^{1/s}
\bigg)_{k=1}^{\infty} \bigg\|_{X}
\]
for all $x_{1}, \dots, x_{N} \in X$.

It is well known that $\ell_{p}$ is $r$ convex for $1 \leq r \leq p$ and $s$-concave for $p \leq s < \infty$.\\

The following result is probably known. However we were not able to find a proper reference of this fact and we include here a short proof. It is modelled along the same lines as the proof of the fact that
if the canonical vectors form  a basis of $X$ then both duals coincide.

\begin{proposition} \label{duales}
If $X$ is a Banach sequence space, its K\"othe dual $X^{\times}$ is a 1-complemented subspace of the usual dual $X^{*}$.
\end{proposition}
\begin{proof} Let us see first that the mapping
$i: X^{\times} \rightarrow X^{*}$ defined by $i(z)=\varphi_{z}:X \rightarrow \mathbb{K}$, with $\varphi_{z}(x)=\sum_{n} z_{n} x_{n}$, is an isometry. It is clearly
well defined; moreover
\[
 \Vert \varphi_{z} \Vert = \sup_{x \in B_{X}} \Big\vert \sum_{n} z_{n} x_{n} \Big\vert \leq  \sup_{x \in B_{X}}  \sum_{n} \vert z_{n} x_{n} \vert
=\Vert z \Vert_{X^{\times}} \, .
\]
To see the reverse inequality, for any $t,s \in \mathbb{K}$ we take $\varepsilon (t,s) \in \mathbb{K}$ with $\vert \varepsilon (t,s) \vert =1$ such that
$\vert s t \vert = \varepsilon (t,s) s t$; then for every $x \in B_{X}$ and every $z \in X^{\times}$ we have
\[
 \sum_{n} \vert z_{n} x_{n} \vert = \sum_{n} \varepsilon ( z_{n}, x_{n}) z_{n} x_{n} = \Big\vert  \sum_{n} \varepsilon ( z_{n}, x_{n}) z_{n} x_{n} \Big\vert
\leq \sup_{a \in B_{X}}  \Big\vert  \sum_{n} z_{n} a_{n} \Big\vert = \Vert \varphi_{z} \Vert \, ,
\]
which gives $\Vert z \Vert_{X^{\times}} \leq \Vert \varphi_{z} \Vert$.

On the other hand, the mapping $q : X^{*} \rightarrow X^{\times}$ given by $q(\varphi) =\big( \varphi (e_{n}) \big)_{n}$ defines a norm-one projection. Indeed,
given $x \in X$ and fixed $N$ we have
\begin{multline*}
 \sum_{n=1}^{N} \vert x_{n} \varphi(e_{n}) \vert =  \sum_{n=1}^{N} \varepsilon (x_{n}, \varphi(e_{n}) ) x_{n} \varphi(e_{n})
= \varphi \Big( \sum_{n=1}^{N} \varepsilon (x_{n}, \varphi(e_{n}) ) x_{n} e_{n} \Big)  \\
\leq \Vert \varphi \Vert \  \Big\Vert \sum_{n=1}^{N} \varepsilon (x_{n}, \varphi(e_{n}) ) x_{n} e_{n} \Big\Vert_{X}
\leq \Vert \varphi \Vert \  \Vert x \Vert \, .
\end{multline*}
This shows that $\sum_{n=1}^{\infty} \vert x_{n} \varphi(e_{n}) \vert \leq  \Vert \varphi \Vert \  \Vert x \Vert$, which gives that $q$ is well defined and $\Vert q (\varphi) \Vert
\leq \Vert \varphi \Vert$. Furthermore, $q$ is a projection, since clearly $q \circ i (z) = q(\varphi_{z}) = (z_{n})_{n}=z$.
\end{proof}

\section{Extension of bilinear forms on Banach sequence spaces}\label{seccion-todos extendibles}
In what follows  $K_{G}$ denotes the  Grothendieck's constant. We begin by proving the following known fact, which was stated as Theorem 3.4 in \cite{JaPaPGVi07}
without the estimates for the norms (see \cite[Lemma 2.4]{CaGaJa01} for a result in the same spirit).

\begin{proposition} \label{sumantes}
If every  bilinear form $B:X \times Y \to \mathbb{K}$
is extendible with $\| B  \|_{\mathcal E}\le K\| B \|$, then every operator $u:X^{*} \to \ell_{2}$ 
is absolutely $1$-summing and
$\pi_{1}(u) \leq K_{G} K \Vert u \Vert$.
\end{proposition}

\begin{proof} We first note that,  by definition of the tensor norm $/\pi \backslash$ (see \cite[Section~20.7]{DeFl93}), $\mathcal{E}^2(X,Y)$ is isometrically the dual of $X\otimes_{/\pi \backslash} Y$. Then, our hypothesis is equivalent to the inequality $\pi \le K /\pi \backslash$ on $X\otimes Y$ and, as a consequence, we also have $\pi\backslash \le K /\pi \backslash$ on $X\otimes Y$. Since both $\pi\backslash$ and $ /\pi \backslash$ are right-injective, an application of Dvoretzky's theorem \cite[19.1]{DiJaTo95} and the previous inequality gives an isomorphism
\begin{equation}\label{pi con palitos}
X\otimes_{/\pi\backslash} \ell_2^N \longrightarrow X\otimes_{\pi\backslash} \ell_2^N
\end{equation}
with norm at most $K$. Since $\ell_2^N$ is finite dimensional,  $\mathcal{L}(X;\ell_{2}^{N})$ and $X^* \otimes \ell_{2}^{N}$
coincide as sets. Then, the embedding in \cite[Section~17.6]{DeFl93} is actually surjective and \cite[Sections~21.5 and 27.2]{DeFl93} give
$X^* \otimes_{w_{\infty}} \ell_{2}^{N} \stackrel{1}{\approx} \Gamma_{\infty}(X;\ell_{2}^{N})$ (see \cite[Chapter~9]{DiJaTo95} or \cite[Section~18]{DeFl93} for the definition of $\Gamma_{\infty}(X;Y)$). Therefore,
\begin{equation}\label{biduales tensores}
\big(X\otimes_{\pi\backslash} \ell_2^N \big)^{**}\stackrel{1}{\approx}\big( \Gamma_{\infty} (X,\ell_2^N) \big) ^{*} \stackrel{1}{\approx}\big(X^*\otimes_{w_{\infty}} \ell_2^N \big)^{*}
\stackrel{1}{\approx} \Pi_{1}(X^{*};\ell_2^N) \, .
\end{equation}
Now, by \cite[Sections~17.12 and 27.2]{DeFl93}, the operator ideal $\Delta_{2}$ is associated to $w_{2}$ and $\Pi_1$ is associated to $ \pi\backslash$. On the other hand
Grothendieck's inequality \cite[Section~20.17]{DeFl93} states that $w_2 \ge K_G  /\pi \backslash$ and clearly we have  $\mathcal L(X^*;\ell_2^N)
\stackrel{1}{\approx} \Delta_2(X^*;\ell_2^N)$. Using \eqref{biduales tensores} to take biduals in \eqref{pi con palitos} we have an isomorphism
\begin{equation}\label{pi con palitos bidual}
\mathcal L(X^*;\ell_2^N)\longrightarrow  \big(X\otimes_{/\pi\backslash}  \ell_2^N \big)^{**} \longrightarrow \big( X^{**}\otimes_{\pi\backslash} \ell_2^N \big)^{**}
\stackrel{1}{\approx} \Pi_1(X^*;\ell_2^N),
\end{equation}
where the first mapping has norm bounded by $K_G$ and the second one by  $K$. Since both $\mathcal L$ and $\Pi_1$ are maximal operator ideals,
the same holds if we put $\ell_2$ instead of $\ell_2^N$.
\end{proof}

\medskip

\noindent With this result we can now prove the following one, from which Theorem~\ref{main1} follows as an immediate consequence.

\begin{theorem} \label{extens con base}
 Let $X$ be a Banach space with an unconditional basis and $Y$ be any infinite dimensional Banach space such that every bilinear form on $X \times Y$ is extendible.
Then $X \approx c_0$.
\end{theorem}
\begin{proof}
Let us see first that, under our assumptions, the basis of $X$ must be shrinking. Suppose it is not. Since it is unconditional,
by James theorem \cite[Corollary~2]{Ja50} (see also \cite[Theorem~3.3.1]{AlKa06}) $X$ must contain a complemented copy of $\ell_1$. Since the property of all bilinear forms
being extendible is inherited by complemented subspaces, it follows that  every bilinear form on $\ell_1\times Y$ is extendible.
This implies \cite[Lemma~6]{KiRy98} that every continuous linear operator from $Y$ to $\ell_\infty$ is absolutely 2-summing. By the so called $\mathfrak L_p$-Local Technique Lemma for Operator
Ideals~\cite[Section~23.1]{DeFl93}, the same holds for every operator from $Y$ to $\ell_\infty(I)$, for any index set $I$.
But this is not possible, since there exists an isometric embedding from $Y$
into some $\ell_\infty(I)$, and this cannot be absolutely 2-summing (otherwise,  the identity on $Y$ would be so, but $Y$ is infinite dimensional).

This means that the
canonical basis of $X$ must be shrinking. We can assume that the basis is 1-unconditional, so that the coordinate basis is an
1-unconditional basis of $X^{*}$. We also know from Proposition~\ref{sumantes} that all operators from $X^*$ to $\ell_2$ are absolutely 1-summing.
By \cite{LiPe68} (see also \cite[Theorem 8.21]{Pi86}) this  implies that the basis of $X^{*}$ is $(K_{G} K)^2$-equivalent to the basis of $\ell_{1}$.
Although $\ell_1$ has many non-isomorphic preduals, if the coordinate basis is equivalent to that of $\ell_1$, a standard computation shows that the canonical basis on $X$ must be $(K_{G} K)^2$-equivalent to the basis of $c_{0}$.
\end{proof}

Note that the proof not only shows that $X$ must be isomorphic to $c_0$, but also gives an estimation of the Banach-Mazur distance between $X$ and $c_0$ whenever $X$ has an 1-unconditional basis.
As a consequence, we can also characterize the pairs of Banach sequence spaces on which every bilinear form is extendible.
\begin{corollary}\label{equivalencias-extendibles}
If $X$ and $Y$ are Banach sequence spaces, then the following are equivalent.
\begin{enumerate}
 \item[(i)] $\mathcal{L}^2( X,Y) = \mathcal{E}^2( X,Y)$  and $\|B\|_{\mathcal{E}} \le K_1\|B\|$ for all bilinear form $B$ on $X\times Y$.

\item[(ii)] The canonical basic sequence of   $X$ and $Y$ are $K_2$-equivalent to the canonical basis of $c_0$.

\item[(iii)] $K_3:=\sup\{d(X_N,\ell_\infty^N),d(Y_N,\ell_\infty^N) : N\in \mathbb N\}$ is finite (where $d$ denotes the Banach-Mazur distance).

\item[(iv)] The spaces $X_N$ and $Y_N$ ($N\in \mathbb N$) are $K_4$-uniformly complemented in some $L_\infty(\mu)$.

\end{enumerate}
Moreover, we have $K_4\le K_3\le K_2\le (K_G\ K_1)^2$ and $K_1 \leq K_2^2\le K_{G}^{4} K_4^8$.
\end{corollary}
\begin{proof}
If (i) holds on $X\times Y$, then the same holds for $X_N\times Y_N$ for any $N$ and, by the density lemma \cite[Section~13.4]{DeFl93}, for $X_0\times Y_0$.  By Theorem~\ref{extens con base},
both bases are $(K_{G}K)^{2}$-equivalent to the basis of $c_{0}$.

The implications
(ii) $\Rightarrow$ (iii) $\Rightarrow$ (iv) are immediate, as well as the inequalities $K_4\leq K_3 \le  K_2$.

If (iv) holds,  bilinear forms on  $X_{N} \times Y_{N}$ are extendible with  $\Vert \cdot \Vert_{\mathcal{E}} \leq K_{4}^{2} \Vert \cdot \Vert $ and, as before, the same holds for $X_{0} \times Y_{0}$. By Theorem~\ref{extens con base} their canonical
bases are $K_G^2K_4^4$-equivalent to the canonical basis of $c_{0}$, which is (ii).

Now suppose (ii) holds and take a bilinear form  $B : X \times Y \rightarrow \mathbb{K}$.
We know from Proposition~\ref{duales} that $X^{\times}$ is 1-complemented in $X^{*}$. Then $(X^{\times})^{*}$ is isometrically a (complemented)
subspace of $X^{**}$. Since (ii) implies $X^\times=\ell_1$, we also have
$(X^{\times})^{*}=X^{\times\times} = \ell_\infty$. The same holds for $Y$, so
we obtain the following diagrams:
\[
\xymatrix{
X  \ar[r]^{i_1} &X^{\times\times}  \ar[r]^{i_2} & X^{**}\\
 & \ell_\infty \ar[u]^{u} & }
 \qquad
\xymatrix{
Y  \ar[r]^{j_1} &Y^{\times\times}  \ar[r]^{j_2} & Y^{**}\\
 & \ell_\infty \ar[u]^{v} & } ,
\]
where $i_1$, $i_2$, $j_1$ and $j_2$ are isometric injections and $u$ and $v$ are isomorphisms with
\begin{equation} \label{uyualamenosuno}
\|u\|\ \|u^{-1}\|\le K_2 \quad\text{ and }\|v\|\ \|v^{-1}\|\le K_2.
\end{equation}
 We can extend $B$ (in the canonical way) to a bilinear form $\tilde{B} : X^{**}\times Y^{**} \rightarrow \mathbb{K}$
with the same norm as $B$, and then define a bilinear form  $\widehat B$ on $\ell_\infty\times \ell_\infty$  by $\widehat B=\tilde B\circ ( i_2 \circ u,  j_2\circ v)$. We have obtained
the factorization $B=\widehat B\circ(u^{-1}\circ i_1,v^{-1}\circ j_1)$.
Since on $\ell_{\infty} \times \ell_{\infty}$ every bilinear form is extendible (with the extendible norm equal to the usual norm), from the ideal property of extendible bilinear
forms and inequalities~\eqref{uyualamenosuno} we conclude that $B$ is extendible and $\|B\|_{\mathcal{E}} \le K_2^2\|B\|$.
\end{proof}

It follows from the previous corollary (and its proof) that a Banach sequence space on which every bilinear form is extendible must satisfy the sublattice inclusions
\begin{equation}\label{sublattice}
c_0\subset X \subset \ell_\infty .
\end{equation}
Conversely, let $X$ be a Banach sequence space satisfying \eqref{sublattice}. By a closed graph argument, both inclusions are continuous and
it is easy to check that $X$ satisfies the equivalent conditions of Corollary~\ref{equivalencias-extendibles}. Note also that a Banach sequence space satisfies \eqref{sublattice}
if and only if its K\"othe dual is $\ell_1$.
As a consequence, we have the following version of Theorem~\ref{main1} for Banach sequence spaces.
\begin{corollary} \label{ultimo}
The Banach sequence spaces $X$ on which every bilinear form is extendible are those satisfying \eqref{sublattice}.
Also, this happens if and only if $X^{\times} = \ell_{1}$.
\end{corollary}

Examples of such spaces are $c_0\oplus \ell_\infty$, $c_{0}(\ell_{\infty})$ and $\ell_{\infty}(c_{0})$. It is not hard to see that these spaces are mutually non-isomorphic Banach sequence spaces (see also~\cite{CeMe10}, where the authors show that $c_{0}(\ell_{\infty})$ and $\ell_{\infty}(c_{0})$ are not isomorphic even as Banach spaces).

If a sequence $X_n$ of $n$-dimensional Banach spaces is uniformly complemented in some $\mathcal L_\infty$, it is an open problem if these spaces have to be
uniformly isomorphic to $\ell_\infty^n$. Taking $X=Y$ in Corollary~\ref{equivalencias-extendibles}, the implication (iv)$\Rightarrow$ (iii) gives the following partial answer.

\begin{proposition}\label{open problem}
If the $N$-dimensional sections $X_N$ of a Banach sequence space $X$ are uniformly complemented
in some $L_\infty(\mu)$, then they must be uniformly isomorphic to $\ell_\infty^N$. \end{proposition}

Note also that, since $\ell_{\infty}$ and $c_{0}$ are the only symmetric Banach sequence spaces satisfying (ii) of Corollary~\ref{equivalencias-extendibles}, these two are the only symmetric Banach sequence spaces on which every bilinear form is extendible.

\bigskip

If $X_1,\dots,X_n$ are Banach spaces such that every $n$-linear form on $X_1\times\cdots\times X_n$ is extendible, then it is  known (and easy to see) that so is every
bilinear form on $X_i\times X_j$  for each pair $i\ne j$. Indeed, given $B \in \mathcal{L}^2(X_i\times X_j)$, we can multiply it by linear functionals to obtain a $n$-linear
form on $X_1\times\cdots\times X_n$. This is extendible by our hypothesis. From this, it is rather immediate to conclude that $B$ is extendible. As a consequence,
multilinear versions of our results follow directly from the bilinear ones.

\medskip

If $X$ and $Y$ are Banach sequence spaces such that every bilinear form on $X\times Y$ is extendible, then we know from
 Theorem~\ref{equivalencias-extendibles} (iii) that  both $X$ and $Y$ contain the $\ell_{\infty}^{N}$ uniformly.
We can extend this statement
to subspaces of Banach lattices.

\begin{proposition} \label{subesp comple}
Let $X_1,X_2$ be subspaces of  Banach lattices such that every $n$-linear form on $X_1,X_2$ is extendible. Then
every infinite dimensional complemented subspace of each $X_{j}$ contains the $\ell_{\infty}^{N}$ uniformly.
\end{proposition}
\begin{proof}
Suppose that there exists a complemented subspace $E$ of $X_1$ that does not contain the $\ell_{\infty}^{N}$ uniformly.
By \cite[Corollary~1]{JoTz75},  $E$ must contain uniformly complemented $N$-dimensional subspaces
$E_{N}$ such that $\sup_{N} d(E_{N}, \ell_{p}^{N}) < \infty$ for $p=1$ or $2$. Since $E$ is complemented in $X_1$, the $E_N$ are also uniformly complemented in $X_1$.
On the other hand, again by \cite[Corollary~1]{JoTz75}, $X_2$ must contain uniformly complemented $N$-dimensional subspaces $F_{N}$ such that  $\sup_{N} d(F_{N}, \ell_{q}^{N}) < \infty$ for $q=1, 2$ or $\infty$.
Our hypotheses ensure that bilinear forms on $X_1\times X_2$ are extendible. Since $E_N$ and $F_N$ are uniformly complemented in $X_1$ and $X_2$ and they are
(uniformly) isomorphic to $\ell_p^N$ and $\ell_q^N$, there must exist $K>0$ such that
$\|B\|_{\mathcal E^2(\ell_p^N, \ell_q^N)}\le K \|B\|_{\mathcal L^2(\ell_p^N, \ell_q^N)}$
for all $N$. Now, the density lemma \cite[Section~13.4]{DeFl93} implies that every bilinear form on $\ell_p\times \ell_q$ (or $\ell_p\times c_0$) must be extendible, which contradicts
Theorem~\ref{extens con base}.
\end{proof}

The converse of Proposition~\ref{subesp comple} does not hold. For example, the Schreier space is $c_0$-saturated and there are non-extendible bilinear forms on it (since there are bilinear forms which are not weakly sequentially continuous). Another conterexample of the converse is~$d_*(w,1)$, the predual of the Lorentz sequence space $d(w,1)$  (see \cite{Sa60} or \cite{Ga66} for a description of the predual).
Since these examples are Banach sequence spaces, they also show that assertion (iii) in
Theorem~\ref{equivalencias-extendibles}  is strictly stronger than containing $\ell_{\infty}^{N}$ uniformly.

\bigskip
In Banach sequence spaces, diagonal bilinear forms are the simplest ones. These are the bilinear forms $T_{\alpha}: X_{1} \times  X_{2} \rightarrow \mathbb{C}$ given by
\[
T_{\alpha}(x^{1}, x^{2}) = \sum_{k=1}^{\infty} \alpha_{k} x^{1}_{k}   x^{2}_{k} \, ,
\]
for some sequence $(\alpha_{k})_{k}$ of scalars.
 We end this note showing, under some assumptions, which are the spaces on which all diagonal bilinear forms are extendible. 

Following standard notation, given a symmetric Banach sequence space $X$ we consider the fundamental function of $X$, given by $\lambda_{X} (N)
:= \big\| \sum_{k=1}^{N} e_{k} \big\|_{X}$ for $N \in \mathbb{N}$.

Given two sequence of real numbers $(a_{n})_{n}$ and $(b_{n})_{n}$ we write $a_{n} \preceq b_{n}$ whenever there is a universal
constant $C>0$ such that $a_{n} \leq C b_{n}$ for every $n$. If $a_{n} \preceq b_{n}$ and $b_{n} \preceq a_{n}$, we write $a_{n} \asymp b_{n}$.

\begin{theorem} \label{2ccv 2cvx}
Let $X$ and $Y$ be symmetric Banach sequence spaces, each being $2$-convex or $2$-concave. Then
all diagonal bilinear forms on $X \times Y$ are extendible if and only if either $X=Y=\ell_{1}$ or $X,Y\in \{c_0,\ell_\infty\}$
\end{theorem}
\begin{proof}
The \textit{if} part is clear: by \cite[Proposition~2.3]{Ca01} (see also \cite[Proposition~1.2]{CaDiSe06})
on $\ell_1$ diagonal bilinear forms are integral (and, therefore, extendible), and in the other cases all bilinear forms are extendible.

For the converse, we consider  the diagonal bilinear form given by
$\phi_N(x,y)= \sum_{i=1}^N x_iy_i$. It is easily computed that $\Vert \phi_N \Vert_{\mathcal{L}^{2} (\ell_2^N)} = 1$; on the other hand, by
\cite[Proposition~1.1]{Ca01} or \cite[Proposition~2.5]{CaGaJa01} we have $\Vert \phi_N \Vert_{\mathcal{E}^{2} (\ell_2^N)} =
\Vert \phi_N \Vert_{\mathcal{N}^{2} (\ell_2^N)} = N$. Let now $\id^{N}_{X} :  \ell_{2}^N \to X_N$ and $\id^{N}_{Y} :  \ell_{2}^N \to Y_N$ be the identity mappings. Comparing the usual and
extendible norms of the bilinear forms $\phi_N$ and $ \phi_N \circ ((\id^{N}_{X})^{-1},(\id^{N}_{Y})^{-1})$, we get
$$N \preceq \Vert \id^{N}_{X} \Vert  \Vert \id^{N}_{Y} \Vert    \Vert (\id^{N}_{X})^{-1}  \Vert (\id^{N}_{Y})^{-1} \Vert.$$

By \cite[16.4]{TJ89} (see also \cite[page~138]{DeMi06}), since $X$ is a symmetric Banach
sequence space we have $d(\ell_{2}^N, X_N) = \Vert \id^{N}_{X} \Vert \ \Vert (\id^{N}_{X})^{-1} \Vert$ (and the same for $Y_N$). Therefore,
\[
 N\preceq  \Vert \id^{N}_{X} \Vert  \Vert \id^{N}_{Y} \Vert    \Vert (\id^{N}_{X})^{-1}  \Vert (\id^{N}_{Y})^{-1} \Vert =  d(\ell_{2}^N, X_N) d(\ell_{2}^N, Y_N) \, .
\]
Since we always have $d(\ell_{2}^N, X_N) \le \sqrt{ N}$ and $d(\ell_{2}^N, Y_N) \le \sqrt{N}$, we can conclude that
$\sqrt{N} \asymp d(\ell_{2}^N, X_N) =  \Vert \id^{N}_{X} \Vert \ \Vert (\id^{N}_{X})^{-1} \Vert$ (and the same for $Y_N$).
We now apply \cite[Lemma~1~(i)]{DeMi06} and get:
\[
\max\Big(\frac{1}{\lambda_{X}(N)},\frac{\lambda_{X}(N)}{N}  \Big) \asymp 1.
\]
From this we can conclude that $X$ must be $\ell_1, c_0$ or $\ell_\infty$. Indeed, suppose we split the natural numbers $\mathbb{N}=I \cup J$, so
that $\big(\frac{1}{\lambda_{X}(N)}\big)_{N\in I} \asymp 1$ and
$\big(\frac{\lambda_{X}(N)}{N}\big)_{N\in J} \asymp 1$. We have then that  $({\lambda_{X}(N)})_{N\in I}$ is bounded and
$(N)_{N\in J} \preceq  (\lambda_{X}(N))_{N\in J}$. Since $({\lambda_{X}(N)})_{N \in \mathbb{N}}$ is non-decreasing, either $I$ or $J$ must be finite. If $J$
is finite, then $({\lambda_{X}(N)})_{N \in \mathbb{N}}$ is bounded and then the norm in $X$ is equivalent to the $\sup$ norm and $X$ is $c_{0}$ or
$\ell_\infty$. If $I$ is finite, then $N \preceq (\lambda_{X}(N))_{N \in \mathbb{N}}$. Although the fundamental sequence of a symmetric Banach sequence space does not
characterize the norm, for this extreme case it is possible to prove that the norm on $X$ must be isomorphic to $\ell_1$: from the estimate
$N \preceq (\lambda_{X}(N))_{N \in \mathbb{N}}$ we easily obtain  $\lambda_{X^\times}(N) \asymp 1$ and, by the previous case, $X^\times$ must be $\ell_\infty$.
Then we have  $X=\ell_1$.
Proceeding in the same way,  $Y$ has to be either $\ell_1, c_0$ or $\ell_\infty$.

It remains to show that on  $c_0\times \ell_1$ and on $\ell_1\times \ell_\infty$ there are non-extendible diagonal bilinear forms. The mapping $c_0\times \ell_1$ given by $(x,x')\mapsto x'(x)$ is the diagonal bilinear form induced by the formal identity. An extension of this mapping to $c_0\times \ell_\infty$ would give a projection from $\ell_\infty $ to $\ell_1$ (see \cite[1.5]{DeFl93}), which does not exist. For $\ell_1\times \ell_\infty$ we can reason in a similar way.
\end{proof}

Both assumptions on symmetry and concavity/convexity in ``only if'' part of the previous theorem cannot be omitted. Indeed, if we take $c_0\oplus\ell_1$ (that seen as a sequence space
is neither symmetric nor $2$-concave or $2$-convex), then every diagonal bilinear form on $(c_0\oplus\ell_1) \times (c_0\oplus\ell_1)$ is the sum of a diagonal bilinear
form on $c_0\times c_0$ and a diagonal bilinear form on $\ell_1\times\ell_1$, and is therefore extendible.

\section*{Acknowledgements}
We wish to thank Andreas Defant for valuable comments and useful conversations that improved the paper.
We also wish to thank the referees,
for their suggestions which improved the final shape of the paper and for pointing to us some  questions and remarks which lead to Corollary~\ref{ultimo} and Proposition~\ref{open problem}.

\end{document}